\documentclass[reqno,a4paper,final]{amsart}

\usepackage[utf8]{inputenc}
\usepackage[T1]{fontenc}
\usepackage[english]{babel}
\usepackage{ifthen}
\usepackage{enumitem}
\usepackage{dsfont}
\usepackage{mathtools}
\usepackage{amssymb}
\usepackage[babel]{csquotes}

\numberwithin{equation}{section} 

\newtheorem{lemma}{Lemma}[section]
\newtheorem{corollary}[lemma]{Corollary}
\newtheorem{proposition}[lemma]{Proposition}
\newtheorem{theorem}[lemma]{Theorem}

\theoremstyle{definition}
\newtheorem{definition}[lemma]{Definition}

\newtheorem{remark}[lemma]{Remark}

\newlist{thm_enum}{enumerate}{1}
\setlist[thm_enum]{label=\normalfont(\alph*)}
\newlist{def_enum}{enumerate}{1}
\setlist[def_enum]{label=\normalfont(\roman*)}
\newlist{equiv_enum}{enumerate}{1}
\setlist[equiv_enum]{label=\normalfont(\roman*)}

\newcommand{\IZ}{\mathbb{Z}}

\newcommand{\IN}{\mathbb{N}}
\newcommand{\IR}{\mathbb{R}}

\newcommand{\abs}[1]{\left\lvert#1\right\rvert}

\newcommand{\norm}[1]{\left\lVert#1\right\rVert}
\newcommand{\normalnorm}[1]{\lVert#1\rVert}
\newcommand{\biggnorm}[1]{\biggl\lVert#1\biggr\rVert}
\newcommand{\bignorm}[1]{\bigl\lVert#1\bigr\rVert}

\newcommand{\R}[2][\empty]{
	\ifthenelse{\equal{#1}{\empty}}
		{\mathcal{R}\left\{#2\right\}}
		{\mathcal{R}_{#1}\left\{#2\right\}}
}

\makeatletter
\newcommand{\LeftEqNo}{\let\veqno\@@leqno}
\makeatother

\renewcommand{\d}{\mathop{}\!d}

\renewcommand{\epsilon}{\varepsilon}

\let\temp\phi
\let\phi\varphi
\let\varphi\temp

\allowdisplaybreaks[4]

\begin{document}

\title[]{Off-Diagonal Sharp Two-Weight Estimates for Sparse Operators}

\begin{abstract}
For a class of sparse operators including majorants of singular integral, square function, and fractional integral operators in a uniform manner, 
we prove off-diagonal two-weight estimates of mixed type in the two-weight and $A_{\infty}$-characteristics.
These bounds are known to be sharp in many cases; as a new result, we prove their sharpness for fractional square functions.
\end{abstract}

\author{Stephan Fackler}
\address{Institute of Applied Analysis, Ulm University, Helmholtzstr.\ 18, 89069 Ulm, Germany}
\email{stephan.fackler@alumni.uni-ulm.de}

\author{Tuomas P.\ Hyt\"onen}
\address{Department of Mathematics and Statistics, P.O.B.\ 68 (Gustaf H\"allstr\"omin katu 2b), FI-00014 University of Helsinki, Finland}
\email{tuomas.hytonen@helsinki.fi}

\thanks{The first author was supported by the DFG grant AR 134/4-1 ``Regularit\"at evolution\"arer Probleme mittels Harmonischer Analyse und Operatortheorie''.  The second author was supported by the ERC Starting Grant ``Analytic-probabilistic methods for borderline singular integrals'' (grant agreement No. 278558), and by the Academy of Finland via the Finnish Centre of Excellence in Analysis and Dynamics Research (project Nos. 271983 and 307333). Parts of the work were done during a research stay of the first author at the University of Helsinki. He thanks the members of the harmonic analysis group for their hospitality.}


\maketitle

\section{Introduction}

	We prove two-weight $L^p_{\sigma} \to L^q_{\omega}$ estimates for \emph{sparse operators}
	\begin{equation}\label{eq:sparse_operator}
		A^{r,\alpha}_{\mathcal{S}}(f) = \biggl( \sum_{Q \in \mathcal{S}} (\abs{Q}^{-\alpha} \int_Q f)^r \mathds{1}_Q \biggr)^{1/r},
	\end{equation}
	where $\mathcal{S}$ is any sparse collection of dyadic cubes as defined below. By now it is known that such $A^{r,\alpha}_{\mathcal{S}}$ dominate many classical operators $T$ in the sense of pointwise estimates of the type
	\begin{equation}\label{eq:sparse_bound}
		\abs{Tf(x)} \lesssim \sum_{j=1}^{N} A_{\mathcal{S}_j}^{r,\alpha} \abs{f}(x),
	\end{equation}
	where the collections $\mathcal S_j$ depend on the function $f$, but the implied constants do not, and so the norm estimates for $A^{r,\alpha}_{\mathcal{S}}$, uniform over the sparse collection $\mathcal S$, imply similar estimates for the corresponding $T$.
	 
	Our main result reads as follows:
	
	\begin{theorem}\label{thm:main}
	Let $1<p\leq q < \infty$, $0<r<\infty$, and $0<\alpha\leq 1$.
	Let $\omega,\sigma\in A_\infty$ be two weights.
	Then $A^{r,\alpha}_{\mathcal S}(\cdot\sigma)$ maps $L^p_\sigma\to L^q_\omega$ if and only if the two-weight $A_{pq}^\alpha$-characteristic
	\begin{equation*}
		[\omega,\sigma]_{A_{pq}^\alpha(\mathcal S)}:=\sup_{Q\in\mathcal S}\abs{Q}^{-\alpha}\omega(Q)^{1/q}\sigma(Q)^{1/p'}
	\end{equation*}
	is finite, and in this case
\begin{equation}\label{eq:mainres}
  1\leq\frac{\norm{A^{r,\alpha}_{\mathcal S}(\cdot\sigma)}_{L^p_\sigma\to L^q_\omega}}{[\omega,\sigma]_{A_{pq}^\alpha(\mathcal S)}}
  \lesssim
  \begin{cases}
     [\sigma]_{A_\infty}^{\frac{1}{q}}+[\omega]_{A_\infty}^{(\frac{1}{r} - \frac{1}{p})_+}, \qquad \text{unless $p=q>r$ and $\alpha<1$}, \\
     [\omega]_{A_{\infty}}^{\frac{1}{r} (1 - \frac{r}{p})^2} [\sigma]_{A_{\infty}}^{\frac{1}{r} ( 1 - (1 - \frac{r}{p})^2)} + [\omega]_{A_{\infty}}^{\frac{1}{r} (1 - (\frac{r}{p})^2)} [\sigma]_{A_{\infty}}^{\frac{1}{r} (\frac{r}{p})^2},
  \end{cases}
\end{equation}
where $x_+:=\max(x,0)$ in the exponent.
	\end{theorem}
	 
	The necessity of the $A_{pq}^\alpha$-condition, and the lower bound in \eqref{eq:mainres}, follows simply by substituting $f=\mathds{1}_Q$ for any $Q\in\mathcal S$ and estimating $A^{r,\alpha}_{\mathcal{S}}(f\sigma)\geq A^{r,\alpha}_{\{Q\}}(f\sigma)=\abs{Q}^{-\alpha}\sigma(Q)\mathds{1}_Q$, so the main point of the theorem is the other estimate.
	 
	 Theorem \ref{thm:main} includes several known cases: (The ``Sobolev'' case $1/p-1/q=1-\alpha$ of these results, together with multilinear extensions, can also be recovered from the recent general framework of \cite{ZK}.)

	 \begin{itemize}
	    \item For $r=\alpha=1$, \eqref{eq:sparse_bound} holds for all Calder\'on--Zygmund operators. The most general version is due to \cite{Lac17}, with a simplified proof in \cite{Ler16}, but its variants go back to \cite{Ler13}. The bound \eqref{eq:mainres} in this case was obtained in \cite{HytPer13} for $p=q=2$ and in \cite{HytLac12} for general $p=q\in(1,\infty)$. 
	    These improved the $A_2$ theorem of \cite{Hyt12b} by replacing a part of the $A_2$ or $A_p$ constant by the smaller $A_\infty$ constant.
	    
	\item For $r=2$ and $\alpha=1$, \eqref{eq:sparse_bound} holds for several square function operators of Littlewood--Paley type \cite{Ler11}. For $p=q$, the mixed bound \eqref{eq:mainres}, even for general $r>0$, is from \cite{LacLi16}; this improves the pure $A_p$ bound of \cite{Ler11}.
	
	    \item For $r=1$ and $0<\alpha<1$, \eqref{eq:sparse_bound} holds for the fractional integral operator
\begin{equation*}
	    I_\gamma f(x)=\int_{\IR^d}\frac{f(y)}{\abs{x-y}^{n-\gamma}}dy,
\end{equation*}
when $\alpha=1-\gamma/n$ \cite{LMPT10}. When also $p<q$, \eqref{eq:mainres} is due to \cite{CruMoe13a}. The ``Sobolev'' case with $1/p-1/q=\gamma/n=1-\alpha\in(0,1)$ was obtained by the same authors in \cite{CruMoe13}, elaborating on the pure $A_{pq}$ bound of \cite{LMPT10}.
Additional complications with $p=q$, which lead to the weaker version of our bound \eqref{eq:mainres}, have been observed and addressed in different ways in \cite{CruMoe13a,CruMoe13}; see also the discussion in \cite[Section 7]{Cruz17}.

	\item For $0<r<1$ and $\alpha=1$, the operators \eqref{eq:sparse_operator} can be related to certain ``rough'' singular integral operators. Namely, several recent works starting with \cite{BFP} have established  sparse domination for ever bigger classes of operators $T$ in the form
\begin{equation*}
  \abs{\langle Tf, g\rangle}\lesssim\sum_{Q\in\mathcal S}\Big(\abs{Q}^{-1}\int_Q\abs{f}^s\Big)^{1/s}\Big(\abs{Q}^{-1}\int_Q\abs{g}^t\Big)^{1/t}\abs{Q},\end{equation*}
for some $s,t\geq 1$;
this is weaker than \eqref{eq:sparse_bound}, but still very powerful for many purposes.
For $s>1=t$, the above domination can be written as
\begin{equation}\label{eq:rough}
   \abs{\langle Tf, g\rangle}\lesssim\langle (A_{\mathcal S}^{1/s,1}\abs{f}^s)^{1/s} , \abs{g} \rangle,
\end{equation}
reducing $L^p$ bounds for $T$ to $L^{p/s}$ bounds for $A_{\mathcal S}^{r,1}$, where $r=1/s\in(0,1)$.
In particular, \cite{CCDO} have proved the bound \eqref{eq:rough} when $T=T_\Omega$ is a homogeneous singular integral with symbol $\Omega\in L_0^\infty(S^{n-1})$; in this case one can take any $s>1$ with implied constant $s'=s/(s-1)$ in \eqref{eq:rough}. Thus
\begin{equation*}
  \|T_\Omega\|_{L^p_\omega\to L^p_\omega}
  \lesssim 
  s'\sup_{\mathcal S}\| A_\mathcal S^{1/s,1}(\cdot\,\sigma_{p/s})\|_{L^{p/s}_{\sigma_{p/s}}\to L^{p/s}_\omega}^{1/s},
\end{equation*}
where $\sigma_{p/s}=\omega^{1-(p/s)'}$ is the $L^{p/s}$ dual weight of $\omega$. From the sharp reverse H\"older inequality for $A_\infty$ weights \cite{HytPer13} one checks, for $s=1+\epsilon_d/[\sigma]_{A_\infty}$, that
 $[\sigma_{p/s}]_{A_\infty}\lesssim[\sigma_p]_{A_\infty}$ and $[\omega,\sigma_{p/s}]_{A_{p/s,p/s}^1}\lesssim[\omega,\sigma_p]_{A_{pp}^1}^s=[\omega]_{A_p}^{s/p}$. With a similar estimate for the adjoint $T_\Omega^*:L^{p'}_{\sigma_p}\to L^{p'}_{\sigma_p}$, this reproduces the bound
 \begin{equation*}
  \|T_\Omega\|_{L^p_\omega\to L^p_\omega}
  \lesssim [\omega]_{A_p}^{1/p}([\omega]_{A_\infty}^{1/p'}+[\sigma_p]_{A_\infty}^{1/p})\min([\omega]_{A_\infty},[\sigma_p]_{A_\infty})\lesssim [\omega]_{A_p}^{p'}
\end{equation*}
from \cite{LPRR}, an elaboration of earlier results for the same operators by \cite{CCDO,HRT17}. (Since we only reconsider known results here, we leave the details for the reader.)
	 \end{itemize}

Certainly, there are also important developments not covered by Theorem \ref{thm:main}. We have used the $A_\infty$ assumption on the weights to bootstrap the generally insufficient two-weight condition $[\omega,\sigma]_{A_{pq}^\alpha(\mathcal S)}<\infty$. Other prominent assumptions found in the literature include:
\begin{itemize}
  \item Orlicz norm bumps, studied among others in \cite{CMP12,Ler13,CruMoe13a,CruMoe13,Cruz17}, with early history in \cite{Neug,Per94b,Per94a};
  \item testing conditions, pioneered in \cite{Saw82,Saw88}, with a recent culmination in \cite{LSSU,LacII}; and
  \item entropy bumps, recently introduced in  \cite{TV:entropy} and studied in \cite{LS:entropy,RS:entropy}.
\end{itemize}
However, an in-depth discussion of any of these topics would take us too far afield. (We will use certain testing conditions as a tool, though.)

While some parameter combinations seem to be new in Theorem \ref{thm:main}, we do not insist too much on this. Our main contribution is the unified approach that covers all these cases at once, without being much longer than the proofs of the existing special cases. On the level of technical details, we build on the previous approach of \cite{HytLi15} to the case $p=q$ and $\alpha=1$. In the particular case of \emph{fractional square functions} (corresponding to $r=2$ and $\alpha\in(0,1)$) we also prove,  in Section~\ref{sec:fractional_square_function}, the sharpness of our estimate for $1/p-1/q=1-\alpha$. The sharpness seems to be new for this class of operators, although the bound itself was already obtained in \cite{ZK}.
	 	
		In the next section we introduce the necessary background material. Afterwards we prove step by step sharp weighted estimates for sparse operators and conclude with the sharpness in the fractional square function case in the final section.

\subsection*{Acknowledgement}
We would like to thank Pavel Zorin-Kranich and an anonymous referee for their friendly suggestions that eliminated several serious omissions in our original overview of related works.

	\section{Preliminaries}
	
		The definition in~\eqref{eq:sparse_operator} is given for sparse families $\mathcal{S}$ of dyadic cubes. Let us be precise about these notions. The \emph{standard dyadic grid} in $\IR^n$ is the collection $\mathcal{D}$ of cubes $\{ 2^{-j} ([0,1)^n + m): j \in \IZ, m \in \IZ^n \}$. For us, a dyadic grid is any family of cubes with similar nestedness and covering properties. Such systems of cubes may be parametrised by  $(\omega_k)_{k \in \IZ} \in (\{ 0, 1 \}^n)^{\IZ}$ as
	\begin{equation*}
		\mathcal{D}^{\omega} \coloneqq \Bigl\{ Q + \sum_{j: 2^{-j} < \ell(Q)} 2^{-j} \omega_j : Q \in \mathcal{D} \Bigr\},
	\end{equation*}
	but we will not need to make use of this explicit representation.
	
	\begin{definition}
		A collection $\mathcal{S}$ of dyadic cubes in $\IR^n$ is called \emph{sparse} if for some $\eta > 0$ there exist pairwise disjoint $(E_Q)_{Q \in \mathcal{S}}$ such that for every $Q \in \mathcal{S}$ the set $E_Q$ is a measurable subset of $Q$ with $\abs{E_Q} \ge \eta \abs{Q}$.
	\end{definition}
	
	A \emph{weight} $\omega$ on $\IR^n$ is a locally integrable function $\omega\colon \IR^n \to \IR_{\ge 0}$. The class of all \emph{$A_{\infty}$-weights} consists of all weights $\omega$ for which their \emph{$A_{\infty}$-characteristic}
	\begin{equation*}
		[\omega]_{A_{\infty}(\IR^n)} \coloneqq \sup_{Q} \frac{1}{\omega(Q)} \int_Q M( \mathds{1}_Q \omega)
	\end{equation*}
	is finite, where $(Mf)(x) \coloneqq \sup_{Q \ni x} \abs{Q}^{-1} \int_Q \abs{f}$ is the Hardy--Littlewood maximal function and where both suprema run over cubes of positive and finite diameter whose sides are parallel to the coordinate axes. 
	
	We introduce some convenient notation. For a positive Borel measure $\sigma\colon \mathcal{B}(\IR^n) \to \IR_{\ge 0}$ with $\sigma(O) > 0$ for all non-empty open subsets $O \subset \IR^n$ and a locally integrable function $f\colon \IR^n \to \IR$ we use the abbreviation $\langle f \rangle_{Q}^{\sigma} = \sigma(Q)^{-1} \int_Q f \d\sigma$ for the mean of $f$ over $Q$ with respect to $\sigma$.

We conclude this section with some remarks about our notion of $A_{pq}^\alpha$:
	
	\begin{remark}
		The usual two-weight $A_p$-characteristic is defined as $[\omega, \sigma]_{A_p} = \sup_Q \abs{Q}^{-p} \omega(Q) \sigma(Q)^{p-1}$. Hence, the relationship to the characteristic used in Theorem~\ref{thm:main} is $[\omega, \sigma]_{A_{pp}^{1}} = [\omega, \sigma]_{A_p}^{1/p}$.
	\end{remark}
	
	Only a limited range of parameters contains non-trivial pairs of weights:
		
	\begin{remark}\label{rem:characteristic}
		The class of weights $\omega, \sigma\colon \IR^n \to \IR_{\ge 0}$ satisfying $[\omega, \sigma]_{A_{pq}^{\alpha}} < \infty$ is only non-empty if $-\alpha + \frac{1}{q} + \frac{1}{p'} \ge 0$. In the diagonal case $p = q$, which maximizes the left hand side because of our standing assumption $p \le q$, this holds if and only if $\alpha \le 1$. Hence, we only get examples for $\alpha \in (0,1]$ and, if $\alpha > \frac{1}{p'}$, for $q \in [p,\frac{p}{p(\alpha-1) + 1}]$. In particular, for $\alpha = 1$ we necessarily are in the diagonal case $p = q$. All this follows from the Lebesgue differentiation theorem by considering cubes centered at and shrinking to a fixed point in $\IR^n$ and the following identity:
		\begin{equation*}
			\abs{Q}^{-\alpha} \biggl( \int_Q \omega \biggr)^{1/q} \biggl( \int_Q \sigma \biggr)^{1/p'} = \abs{Q}^{-\alpha + \frac{1}{q} + \frac{1}{p'}} \biggl( \frac{1}{\abs{Q}} \int_Q \omega \biggr)^{1/q} \biggl( \frac{1}{\abs{Q}} \int_Q \sigma \biggr)^{1/p'}.
		\end{equation*}
		Note that $[\mathds{1}, \mathds{1}]_{A_{pq}^{\alpha}} < \infty$ if and only if $-\alpha + \frac{1}{q} + \frac{1}{p'} = 0$, or equivalently, $\alpha = 1 + \frac{1}{q} - \frac{1}{p}$. Further, in the case $-\alpha + \frac{1}{q} + \frac{1}{p'} > 0$ one can verify that there exist $(\beta, \gamma) \in (0,n) \times (0,n)$ such that $\omega(t) = \abs{t}^{-\beta}$ and $\sigma(t) = \abs{t}^{-\gamma}$ satisfy $[\omega, \sigma]_{A_{pq}^{\alpha}} < \infty$.
	\end{remark}

	\section{Testing Constant Type Estimates for the Operator Norm}
	
	As a first central step we prove the following estimate for the operator norm of $A_{\mathcal{S}}^{r,\alpha}$ between two weighted $L^p$-spaces in terms of two testing constants.
	
	\begin{proposition}\label{prop:testing}
		Let $1 < p \le q < \infty$, $r \in (0, \infty)$, $\alpha \in (0,1]$, $\mathcal{S}$ a sparse collection of dyadic cubes and let $\omega, \sigma\colon \IR^n \to \IR_{\ge 0}$ be weights. Define the testing constants
		\begin{align*}
			\mathcal{T} & = \sup_{R \in \mathcal{S}} \sigma(R)^{-r/p} \bignorm{\sum_{Q \in \mathcal{S}: Q \subset R} \abs{Q}^{-\alpha r} \sigma(Q)^r \mathds{1}_Q}_{L^{q/r}_{\omega}}, \\
			\mathcal{T}^* & = \sup_{R \in \mathcal{S}} \omega(R)^{-1/(q/r)'} \bignorm{\sum_{Q \in \mathcal{S}: Q \subset R} \abs{Q}^{-\alpha r} \sigma(Q)^{r-1} \omega(Q) \mathds{1}_Q}_{L^{(p/r)'}_{\sigma}}.
		\end{align*}
		Then
		\begin{equation*}
			\norm{A^{r,\alpha}_{\mathcal{S}}(\sigma \cdot)}_{L^p_{\sigma} \to L^q_{\omega}}^r \lesssim
			\begin{cases}
				\mathcal{T} + \mathcal{T}^* & \text{if } r < p, \\
				\mathcal{T} & \text{if } r \ge p.
			\end{cases}
		\end{equation*}
	\end{proposition}

	As preparatory steps for the proof of the above result we show some intermediate results. In a first step we reduce the norm estimate for $A^{r,\alpha}_{\mathcal{S}}$ to an estimate for a linear operator. This reduction does not make use of the concrete form of $A^{r,\alpha}_{\mathcal{S}}$. We therefore formulate it in a more general setting.
	
	\begin{lemma}\label{lem:linearize_square_function}
		Let $1 < r < p \le q < \infty$, $\mathcal{C}$ a collection of dyadic cubes and $\omega, \sigma\colon \IR^n \to \IR_{\ge 0}$ be weights. Then for non-negative real numbers $(c_Q)_{Q \in \mathcal{C}}$, the expressions
		\begin{align*}
			I:=&\sup_{f \ge 0, \norm{f}_{L^p_{\sigma}} \le 1} \biggnorm{\sum_{Q \in \mathcal{C}} c_Q \biggl(\int_Q f \sigma \biggr)^r \mathds{1}_Q}_{L^{q/r}_{\omega}},\qquad\text{and} \\
			II:=& \sup_{g \ge 0, \norm{g}_{L^{p/r}_{\sigma}} \le 1} \biggnorm{\sum_{Q \in \mathcal{C}} c_Q \sigma(Q)^{r-1} \abs{Q} \langle g \sigma \rangle_Q \mathds{1}_Q}_{L^{q/r}_{\omega}}
		\end{align*}
		are comparable with constants independent of the concrete choice of $(c_Q)$, $\omega$ and $\sigma$.
	\end{lemma}
	
	\begin{proof}
	This follows from the estimates
	\begin{equation*}
\begin{split}
   \Big(\int_Q f\sigma\Big)^r
   &\leq\sigma(Q)^{r-1}\int_Q f^r\sigma
   =\sigma(Q)^{r-1}\abs{Q}\langle f^r\sigma\rangle_Q \\
   &=\sigma(Q)^r\frac{1}{\sigma(Q)}\int_Q f^r\sigma 
   \leq \sigma(Q)^r(\inf_{x\in Q}M_{\sigma,r} f)^r,\quad M_{\sigma,r}f:= (M_{\sigma}\abs{f}^r)^{1/r}, \\
   &\leq \Big(\int_Q (M_{\sigma,r} f) \sigma\Big)^r,
\end{split}
\end{equation*}
observing that $\norm{f^r}_{L^{p/r}_\sigma}^{1/r}=\norm{f}_{L^p_\sigma}\eqsim\norm{M_{\sigma,r} f}_{L^p_\sigma}$ when $p>r$.
	\end{proof}
	
	For the complement parameter range for $r$, a second argument for the domination of the operator norm by the testing constants is needed. This can again be proven in a more general context.
	
	\begin{lemma}\label{lem:reduction_principal_cubes}
		Let $1 < p \le 	q < \infty$, $r \in [p, \infty)$, $\mathcal{S}$ a sparse collection of dyadic cubes and let $\omega, \sigma\colon \IR^n \to \IR_{\ge 0}$ be weights. For non-negative real numbers $(c_Q)_{Q \in \mathcal{\mathcal{S}}}$ and measurable $f\colon \IR^n \to \IR_{\ge 0}$ we have
			\begin{align*}
				\biggnorm{\sum_{Q \in \mathcal{S}} c_Q \biggl( \int_Q f \sigma \biggr)^{r} \mathds{1}_Q}_{L^{q/r}_{\omega}} \lesssim \sup_{R \in \mathcal{S}} \sigma(R)^{-r/p} \bignorm{\sum_{Q \in \mathcal{S}: Q \subset R} c_Q \sigma(Q)^r \mathds{1}_Q}_{L^{q/r}_{\omega}} \cdot \norm{f}^r_{L^p_{\sigma}}.
			\end{align*}
	\end{lemma}
	\begin{proof}
		It suffices to show the estimate for sparse collections $\mathcal{S}$ of cubes whose side lengths are bounded from above by a fixed constant. In the following we use the principal cubes associated to $f$ and $\sigma$: consider $\mathcal{F} = \cup_{k=0}^{\infty} \mathcal{F}_k$ for $\mathcal{F}_0 = \{ \text{maximal cubes in } \mathcal{S} \}$ and $\mathcal{F}_{k+1} \coloneqq \cup_{F \in \mathcal{F}_k} \operatorname{ch}_{\mathcal{F}}(F)$, where $\operatorname{ch}_{\mathcal{F}}(F) \coloneqq \{ \mathcal{S} \ni Q \subsetneq F \text{ maximal with } \langle f \rangle_Q^{\sigma} > 2 \langle f \rangle_F^{\sigma} \}$. Further, for $Q \in \mathcal{S}$ we write $\pi(Q)$ for the minimal cube in $\mathcal{F}$ containing $Q$. Using the principal cubes, we obtain
		\begin{align*}
				\MoveEqLeft \biggnorm{\biggl( \sum_{Q \in \mathcal{S}} c_Q \biggl( \int_Q f \sigma \biggr)^{r} \mathds{1}_Q \biggr)^{1/r}}_{L^q_{\omega}} \\
				& = \biggnorm{\biggl( \sum_{Q \in \mathcal{S}} c_Q (\langle f \rangle_Q^{\sigma})^r \sigma(Q)^r \mathds{1}_Q \biggr)^{1/r}}_{L^q_{\omega}} \\
				& \le 2 \biggnorm{\biggl( \sum_{F \in \mathcal{F}} (\langle f \rangle_F^{\sigma})^r \sum_{Q: \pi(Q) = F} \sigma(Q)^r c_Q \mathds{1}_Q \biggr)^{1/r}}_{L^q_{\omega}}.
		\end{align*}
		Here we use that for $F = \pi(Q)$ one has $\langle f \rangle_Q^{\sigma} \le 2 \langle f \rangle_F^{\sigma}$. In fact, if $\langle f \rangle_Q^{\sigma} > 2 \langle f \rangle_F^{\sigma}$ would hold, then $F \supsetneq F' \supseteq Q$ for some $F' \in \operatorname{ch}_{\mathcal{F}}(F)$ by the maximality property of $\operatorname{ch}_{\mathcal{F}}(F)$. But then $F' \in \mathcal{F}$ satisfies $F' \supseteq Q$ and is strictly smaller than $F$, which contradicts the minimality of $F = \pi(Q)$.  Because of $p \le r$ the norm on the right hand side is dominated by
		\begin{align*}
			\MoveEqLeft \biggnorm{\biggl( \sum_{F \in \mathcal{F}} (\langle f \rangle_F^{\sigma})^p \bigl( \sum_{Q: \pi(Q) = F} \sigma(Q)^r c_Q \mathds{1}_Q \bigl)^{p/r} \biggr)^{1/p}}_{L^q_{\omega}} \\
			& = \biggnorm{\sum_{F \in \mathcal{F}} (\langle f \rangle_F^{\sigma})^p \Big( \sum_{Q: \pi(Q) = F} \sigma(Q)^r c_Q \mathds{1}_Q \Big)^{p/r}}_{L^{q/p}_{\omega}}^{1/p} \\
			& \le \biggl( \sum_{F \in \mathcal{F}} (\langle f \rangle^{\sigma}_F)^p \biggnorm{\Big( \sum_{Q: \pi(Q) = F} \sigma(Q)^r c_Q \mathds{1}_Q \Big)^{1/r}}_{L^q_{\omega}}^p \biggr)^{1/p}.
		\end{align*}
		Note that in the last step we can use the triangle inequality because of the assumption $p \le q$. Furthermore, the last expression to the power $r$ is dominated by
		\begin{equation*}
			\sup_{R \in \mathcal{S}} \sigma(R)^{-r/p} \bignorm{\sum_{Q \in \mathcal{S}: Q \subset R} c_Q \sigma(Q)^r \mathds{1}_Q}_{L^{q/r}_{\omega}} \cdot \biggl( \sum_{F \in \mathcal{F}} (\langle f \rangle_F^{\sigma})^p \sigma(F) \biggr)^{r/p}.
		\end{equation*}
		We estimate the last factor. Let $x \in \IR^n$ be fixed. If $x$ lies in some cube in $\mathcal{F}_0$, we let $Q_0$ be the unique cube in $\mathcal{F}_0$ that contains $x$. Further, if $x$ lies in some cube in $\operatorname{ch}_{\mathcal{F}}(Q_0)$, we let $Q_1$ be the unique cube in $\operatorname{ch}_{\mathcal{F}}(Q_0)$ with $x \in Q_1$. Continuing inductively, we obtain a possibly finite or even empty chain of cubes $Q_0, Q_1, \ldots$, each of them containing $x$. Let $N \in \IN$ with $x \in Q_N$. Then
		\begin{equation*}
			\sum_{k=0}^{N} (\langle f \rangle_{Q_k}^{\sigma})^p \le \sum_{k=0}^N 2^{-Np} (\langle f \rangle_{Q_N}^{\sigma})^p \lesssim (M_{\sigma} f)^p(x).
		\end{equation*}
		Since $x \in \IR^n$ and $N$ is arbitrary, we get the pointwise domination $\sum_{F \in \mathcal{F}} (\langle f \rangle_{F}^{\sigma})^p \mathds{1}_F \lesssim (M_{\sigma} f)^p$. In particular, we have
		\begin{equation*}
			\sum_{F \in \mathcal{F}} (\langle f \rangle_F^{\sigma})^p \sigma(F) = \int_{\IR^n} \sum_{F \in \mathcal{F}} (\langle f \rangle_F^{\sigma})^p \mathds{1}_F \d \sigma \lesssim \int_{\IR^n} (M_{\sigma} f)^p \d \sigma \lesssim \norm{f}_{L^p(\sigma)}^p. \qedhere
		\end{equation*}
	\end{proof}
	
	As a central blackbox result we need the following norm characterization proved by Lacey, Sawyer and Uriarte-Tuero~\cite[Theorem~1.11]{LacSawUri10}.
	
	\begin{lemma}\label{lem:model_estimate}
		Let $1 < p \le q < \infty$ and let $\omega, \sigma\colon \IR^n \to \IR_{\ge 0}$ be weights. For a collection of dyadic cubes $\mathcal{D}$ and non-negative real $(\tau_{Q})_{Q \in \mathcal{D}}$ consider
			\begin{equation*}
				T(f) = \sum_{Q \in \mathcal{D}} \tau_Q \langle f \rangle_Q \mathds{1}_Q.
			\end{equation*}
		Then 
		\begin{align*}
			\MoveEqLeft \norm{T(\cdot \sigma)}_{L^p_{\sigma} \rightarrow L^q_{\omega}} \simeq \sup_{R \in \mathcal{D}} \omega(R)^{-1/q'} \biggnorm{\sum_{Q \in \mathcal{D}: Q \subset R} \tau_Q \langle \omega \rangle_Q \mathds{1}_Q}_{L^{p'}_{\sigma}} \\
			& + \sup_{R \in \mathcal{D}} \sigma(R)^{-1/p} \biggnorm{\sum_{Q \in \mathcal{D}: Q \subset R} \tau_Q \langle \sigma \rangle_Q \mathds{1}_Q}_{L^q_{\omega}}.
		\end{align*}
	\end{lemma}
	
	Our preparations for the proof of Proposition~\ref{prop:testing} are now completed.
		
	\begin{proof}[Proof of Proposition~\ref{prop:testing}]
		Let us first consider the case $r < p \le q < \infty$. We apply Lemma~\ref{lem:linearize_square_function} for the choice $c_Q = \abs{Q}^{-\alpha r}$ which reduces the estimate to an estimate for $T(\cdot \sigma) \colon L^{p/r}_{\sigma} \to L^{q/r}_{\omega}$, where
		\begin{equation*}
			T\colon f \mapsto \sum_{Q \in \mathcal{S}} \abs{Q}^{1 -\alpha r} \sigma(Q)^{r-1} \langle f \rangle_Q \mathds{1}_Q.
		\end{equation*}
		Now, with the choice $\tau_Q = \abs{Q}^{1-\alpha r} \sigma(Q)^{r-1}$ we are in the setting of Lemma~\ref{lem:model_estimate}. Putting everything together, we get
		\begin{align*}
			\MoveEqLeft \norm{A^{r,\alpha}_{\mathcal{S}}(\cdot \sigma)}_{L^p_{\sigma} \to L^q_{\omega}} \\
			& \simeq \sup_{R \in \mathcal{S}} \omega(R)^{-1/(q/r)'} \biggnorm{\sum_{Q \in \mathcal{S}: Q \subset R} \abs{Q}^{1-\alpha r} \sigma(Q)^{r-1} \abs{Q}^{-1} \omega(Q) \mathds{1}_Q}_{L^{(p/r)'}_{\sigma}} \\
			& + \sup_{R \in \mathcal{S}} \sigma(R)^{-r/p} \biggnorm{\sum_{Q \in \mathcal{S}: Q \subset R} \abs{Q}^{1-\alpha r} \sigma(Q)^{r-1} \abs{Q}^{-1} \sigma(Q) \mathds{1}_Q}_{L^{q/r}_{\omega}}. 
		\end{align*}
		Secondly, the case $r \ge p$ follows from Lemma~\ref{lem:reduction_principal_cubes} with $c_Q = \abs{Q}^{-\alpha r}$.
	\end{proof}
	
	\section{Sharp Estimates for the Testing Constants}
	
	In the last section we saw that it suffices to control the size of the associated testing constants $\mathcal{T}$ and $\mathcal{T}^*$ instead of the operator norm of $A^{r,\alpha}_{\mathcal{S}}$. In this section we establish sharp estimates for these constants. We again need some preparatory steps before coming to the main result.
	We borrow the following lemma from \cite[Proposition~2.2]{CasOrtVer04}.

	\begin{lemma}\label{lem:norm_dyadic}
		Let $p \in (1, \infty)$ and $\sigma\colon \mathcal{B}(\IR^n) \to \IR_{\ge 0}$ be a positive Borel measure with $\sigma(O) > 0$ for all non-empty open $O \subset \IR^n$. For a collection of dyadic cubes $\mathcal{D}$ and non-negative $(\alpha_Q)_{Q \in \mathcal{D}}$ set
		\begin{equation*}
			\phi = \sum_{Q \in \mathcal{D}} \alpha_Q \mathds{1}_Q, \qquad \phi_Q = \sum_{Q' \in \mathcal{D}: Q' \subset Q} \alpha_{Q'} \mathds{1}_{Q'}.
		\end{equation*}
		Then
		\begin{equation*}
			\norm{\phi}_{L^p_{\sigma}} \simeq \biggl( \sum_{Q \in \mathcal{D}} \alpha_Q (\langle \phi_Q \rangle_Q^{\sigma})^{p-1} \sigma(Q) \biggr)^{1/p}.
		\end{equation*}
	\end{lemma}
	
	The first part of the following lemma can be deduced from the special case shown in~\cite[Lemma~5.2]{Hyt14}. However, we prefer to give a direct proof.
	
	\begin{lemma}\label{lem:measure_estimate}
		Let $\omega, \sigma\colon \IR^n \to \IR_{\ge 0}$ be weights and $\alpha, \beta, \gamma \ge 0$ with $\alpha + \beta + \gamma \ge 1$. For a sparse family $\mathcal{S}$ of cubes and a cube $R$ the following holds.
		\begin{thm_enum}
			\item\label{lem:measure_estimate:non_constant} For $\alpha > 0$ one has the universal estimate
				\begin{equation*}
					\sum_{Q \in \mathcal{S}: Q \subset R} \abs{Q}^{\alpha} \sigma(Q)^{\beta} \omega(Q)^{\gamma} \lesssim \abs{R}^{\alpha} \sigma(R)^{\beta} \omega(R)^{\gamma}.
				\end{equation*}
			\item\label{lem:measure_estimate:constant} For $\alpha = 0$ one still has the weaker inequality
				\begin{equation*}
					\sum_{Q \in \mathcal{S}: Q \subset R} \sigma(Q)^{\beta} \omega(Q)^{\gamma} \lesssim [\sigma]_{A_{\infty}}^{\beta} [\omega]_{A_{\infty}}^{\gamma} \sigma(R)^{\beta} \omega(R)^{\gamma}.
				\end{equation*}
		\end{thm_enum}
	\end{lemma}
	\begin{proof}
		For both parts it suffices to treat the case $\alpha + \beta + \gamma = 1$. In fact, if $\alpha + \beta + \gamma \ge 1$, let $\delta = (\alpha + \beta + \gamma)^{-1} \le 1$. Then, for example for part~\ref{lem:measure_estimate:non_constant}, one has
		\begin{align*}
			\MoveEqLeft \sum_{Q \in \mathcal{S}: Q \subset R} \abs{Q}^{\alpha} \sigma(Q)^{\beta} \omega(Q)^{\gamma} \le \biggl( \sum_{Q \in \mathcal{S}: Q \subset R} \abs{Q}^{\delta \alpha} \sigma(Q)^{\delta \beta} \omega(Q)^{\delta \gamma} \biggr)^{1/\delta} \\
			& \lesssim \bigl( \abs{R}^{\alpha \delta } \sigma(R)^{\beta \delta} \omega(R)^{\gamma \delta} \bigr)^{1/\delta} = \abs{R}^{\alpha} \sigma(R)^{\beta} \omega(R)^{\gamma}.
		\end{align*}
		Now let $\alpha + \beta + \gamma = 1$. Let us start with the proof of part~\ref{lem:measure_estimate:non_constant}. We have the estimate
		\begin{align*}
				\MoveEqLeft \sum_{Q \in \mathcal{S}: Q \subset R} \abs{Q}^{\alpha} \sigma(Q)^{\beta} \omega(Q)^{\gamma} = \sum_{Q \in \mathcal{S}: Q \subset R} \abs{Q}^{\alpha + \beta + \gamma} \biggl( \frac{1}{\abs{Q}} \int_Q \sigma \biggr)^{\beta} \biggl( \frac{1}{\abs{Q}} \int_Q \omega \biggr)^{\gamma} \\
				& \le \sum_{Q \in \mathcal{S}: Q \subset R} \abs{Q} \inf_{Q} M(\sigma \mathds{1}_R)^{\beta} \inf_{Q} M(\omega \mathds{1}_R)^{\gamma}.
			\end{align*}
		Since $\beta + \gamma < 1$, there exists some $p \in (1, \infty)$ with $p \beta < 1$ and $p' \gamma < 1$. Using H\"older's inequality, the above term is dominated by
		\begin{align*}
			\MoveEqLeft \biggl( \sum_{Q \in \mathcal{S}: Q \subset R} \abs{Q} \inf_{Q} M(\sigma \mathds{1}_R)^{p\beta} \biggr)^{1/p} \biggl( \sum_{Q \in \mathcal{S}: Q \subset R} \abs{Q} \inf_{Q} M(\omega \mathds{1}_R)^{p' \gamma} \biggr)^{1/p'}.
		\end{align*}
		We now estimate the first term in brackets. Using the sparseness of $\mathcal{S}$ and $\norm{f}_{L^{1,\infty}} \simeq \sup_{\abs{E} \in (0, \infty)} \abs{E}^{1-\frac{1}{p\beta}} ( \int_E \abs{f}^{p\beta} )^{1/(p\beta)}$, this term is dominated up to a universal constant by
		\begin{align*}
			\MoveEqLeft \biggl( \sum_{Q \in \mathcal{S}: Q \subset R} \abs{E(Q)} \inf_{Q} M(\sigma \mathds{1}_R)^{p\beta} \biggr)^{1/p} \le \biggl( \int_{R} M(\sigma \mathds{1}_R)^{p \beta} \biggr)^{1/p} \\
			& \lesssim \norm{M(\sigma \mathds{1}_R)}_{L^{1, \infty}}^{\beta} \abs{R}^{1/p-\beta} \lesssim \norm{\sigma \mathds{1}_R}_{L^1}^{\beta} \abs{R}^{1/p-\beta} = \sigma(R)^{\beta} \abs{R}^{1/p - \beta}.
		\end{align*}
		Putting all the estimates together, we obtain the asserted inequality. For part~\ref{lem:measure_estimate:constant} observe that by Hölder's inequality
		\begin{align*}
			\sum_{Q \in \mathcal{S}: Q \subset R} \sigma(Q)^{\beta} \omega(Q)^{\gamma} & \le \biggl( \sum_{Q \in \mathcal{S}: Q \subset R} \sigma(Q) \biggr)^{\beta} \biggl( \sum_{Q \in \mathcal{S}: Q \subset R} \omega(Q) \biggr)^{\gamma} \\
			& \lesssim [\sigma]_{A_{\infty}}^{\beta} [\omega]_{A_{\infty}}^{\gamma} \sigma(R)^{\beta} \omega(R)^{\gamma}.
		\end{align*}
		The inequalities used in the last estimate can be seen as follows (for $\omega$):
		\begin{equation*}
			\sum_{Q \in \mathcal{S}: Q \subset R} \int_Q \omega \le \sum_{Q \in \mathcal{S}: Q \subset R} \abs{Q} \inf_{E(Q)} M(\mathds{1}_Q \omega) \lesssim \int_{R} M(\mathds{1}_Q \omega) \le[\omega]_{{A}_{\infty}} \omega(R). \qedhere
		\end{equation*}
	\end{proof}
	
	The next proposition is the heart of the article. Here we prove the required estimates on the testing constants $\mathcal{T}$ and $\mathcal{T}^*$. Note the emergence of additional factors in the diagonal case for fractional sparse operators. 
	
	\begin{theorem}\label{thm:two_weight_sparse}
		Let $\omega, \sigma\colon \IR^n \to \IR_{\ge 0}$ be $A_{\infty}$-weights, $\mathcal{S}$ a sparse family of dyadic cubes, $r \in (0, \infty)$ and $\alpha \in (0,1]$. For $1 < p \le q < \infty$ with $-\alpha + \frac{1}{q} + \frac{1}{p'} \ge 0$ one has
		\begin{equation*}
			\mathcal{T} \lesssim \begin{cases}
				[\omega, \sigma]_{A_{pq}^{\alpha}}^r [\sigma]_{A_{\infty}}^{1-(1-r/p)^2} [\omega]_{A_{\infty}}^{(1-r/p)^2} & \text{if } p = q \text{ and } \alpha < 1 \text{ and } p > r, \\
				[\omega,\sigma]_{A_{pq}^{\alpha}}^{r} [\sigma]_{A_{\infty}}^{r/q} & \text{else.}
			\end{cases}
		\end{equation*}
		Further, if $p > r$, the second testing constant satisfies
		\begin{equation*} 
			\mathcal{T^*} \lesssim \begin{cases}
				[\omega, \sigma]_{A_{pq}^{\alpha}}^{r} [\omega]_{A_{\infty}}^{1-(r/p)^2} [\sigma]_{A_{\infty}}^{(r/p)^2} & \text{if } p = q \text{ and } \alpha < 1, \\
				[\omega, \sigma]_{A_{pq}^{\alpha}}^{r} [\omega]_{A_{\infty}}^{1-r/ p} & \text{else.}
			\end{cases}
		\end{equation*}
	\end{theorem}
	\begin{proof}
		
		Throughout the proof, we write $[\omega,\sigma]:=[\omega,\sigma]_{A_{pq}^\alpha}$ for brevity.
		
		\textbf{Part I:} We begin with the estimate for the testing constant $\mathcal{T}$. Here we start with the special case $q > r$. By Lemma~\ref{lem:norm_dyadic} and $\frac{q}{r} > 1$ we rewrite the norm as
    	\begin{align*}
    		\MoveEqLeft \bignorm{\sum_{Q \in \mathcal{S}: Q \subset R} \abs{Q}^{-\alpha r} \sigma(Q)^r \mathds{1}_Q}_{L^{q/r}_{\omega}} \\
    		& \simeq \biggl( \sum_{Q \in \mathcal{S}: Q \subset R} \abs{Q}^{-\alpha r} \sigma(Q)^r \omega(Q) \\
			& \qquad \cdot \Big[\omega(Q)^{-1} \sum_{Q' \in \mathcal{S}: Q' \subset Q} \abs{Q'}^{-\alpha r} \sigma(Q')^r \omega(Q') \Big]^{q/r - 1} \biggr)^{r/q}.
    	\end{align*}
		As a first step we estimate the inner sum (from now on omitting $Q' \in \mathcal{S}$)
		\begin{equation*}
			\sum_{Q' \subset Q} \abs{Q'}^{-\alpha r} \sigma(Q')^r \omega(Q').
		\end{equation*}
		For this let $\delta > 0$ be arbitrary. For each summand we have the estimate
		\begin{equation}
			\label{eq:estimate_use_two_weight}
			\begin{split}
				\MoveEqLeft \abs{Q'}^{-\alpha r} \sigma(Q')^r \omega(Q') \\
				& = (\abs{Q'}^{-\alpha} \sigma(Q')^{1/p'} \omega(Q')^{1/q})^{\delta} \abs{Q'}^{\alpha(\delta - r)} \sigma(Q')^{r-\delta/p'} \omega(Q')^{1-\delta/q} \\
				& \le [\omega, \sigma]^{\delta} \abs{Q'}^{\alpha(\delta - r)} \sigma(Q')^{r-\delta/p'} \omega(Q')^{1-\delta/q}.
			\end{split}
		\end{equation}
		We want to use Lemma~\ref{lem:measure_estimate}. Its assumptions imply restrictions on the possible range of $\delta$. In fact, the following four conditions must be satisfied:
		\begin{equation}\label{eq:delta_assumptions}
			\begin{split}
				\alpha (\delta-r) \ge 0 & \Leftrightarrow \delta \ge r \\
				r-\frac{\delta}{p'} \ge 0 & \Leftrightarrow \delta \le rp' \\
				1-\frac{\delta}{q} \ge 0 & \Leftrightarrow \delta \le q \\
				\alpha(\delta-r) + r - \frac{\delta}{p'} + 1 - \frac{\delta}{q} \ge 1 & \Leftrightarrow \delta \left(\alpha - \frac{1}{p'} - \frac{1}{q} \right) + r(1-\alpha) \ge 0.
			\end{split}
		\end{equation}
		The first three conditions imply that $\delta \in [r, \min(rp',q)]$ which has non-empty interior by our assumption $q > r$. The restriction imposed by the last inequality depends on the parameters. First, if $\alpha - \frac{1}{p'} - \frac{1}{q} = 0$, the last condition holds because of $\alpha \le 1$. 
		Secondly, if $\alpha < \frac{1}{p'} + \frac{1}{q}$ (this implies $\alpha < 1$), we have
		\begin{equation*}
			\delta \le r \frac{1 - \alpha}{\frac{1}{p'} + \frac{1}{q} - \alpha}.
		\end{equation*}
		Note that because of $\frac{1}{p'} + \frac{1}{q} \le 1$ the fraction on the right hand side is bigger or equal to $1$. Hence, if $q > p$, we can find a $\delta$ with $\delta > r$ satisfying all conditions in~\eqref{eq:delta_assumptions}. However, in the case $p = q$ the only possible choice is $\delta = r$. This is the only case where the strict inequality $\delta > r$ cannot be achieved. Summarizing our findings, Lemma~\ref{lem:measure_estimate} gives the estimate
		\begin{align*}
			\MoveEqLeft \sum_{Q' \subset Q} \abs{Q'}^{-\alpha r} \sigma(Q')^r \omega(Q') \\
			& \lesssim [\omega, \sigma]^{\delta} \abs{Q}^{\alpha(\delta-r)} \sigma(Q)^{r-\delta/p'} \omega(Q)^{1-\delta/q} \cdot\begin{cases}
				[\sigma]_{A_{\infty}}^{r/p} [\omega]_{A_{\infty}}^{1-r/p} & p = q \text{ and } \alpha < 1, \\
				1 & \text{else}.
			\end{cases}
		\end{align*}
    	We now show estimates for arbitrary $\delta > 0$ satisfying the inequalities in~\eqref{eq:delta_assumptions}. In the following we will ignore the additional factors in the case $p = q$ and $\alpha < 1$. Using the estimate for the inner sum we have
    	\begin{align*}
    		\MoveEqLeft \bignorm{\sum_{Q \in \mathcal{S}: Q \subset R} \abs{Q}^{-\alpha r} \sigma(Q)^r \mathds{1}_Q}_{L^{q/r}_{\omega}} \\
    		& \lesssim [\omega,\sigma]^{\delta(1-r/q)} \biggl( \sum_{Q \in \mathcal{S}: Q \subset R} \abs{Q}^{-\alpha r} \sigma(Q)^r \omega(Q) \\
			& \qquad \cdot \bigl(\omega(Q)^{-1} \sigma(Q)^{r-\delta/p'} \abs{Q}^{\alpha(\delta-r)} \omega(Q)^{1-\delta/q} \bigr)^{q/r - 1} \biggr)^{r/q} \\
    		& = [\omega,\sigma]^{\delta(1-r/q)} \biggl( \sum_{Q \in \mathcal{S}: Q \subset R} \abs{Q}^{\alpha(-q + \delta q/r - \delta)} \\
			& \qquad \cdot \omega(Q)^{1+\delta/q-\delta/r} \sigma(Q)^{q-\delta/p'(q/r-1)} \biggr)^{r/q}.
    	\end{align*}
    	We pull another power of the two-weight constant out of the sum using exactly the power for which the sum becomes independent of the weight $\omega$. Explicitly, this gives
    	\begin{align*}
    		\MoveEqLeft \abs{Q}^{-\alpha q + \delta \alpha q/r - \delta \alpha} \omega(Q)^{1 + \delta/q - \delta / r} \sigma(Q)^{q-\delta/p'(q/r-1)} \\
			& = \bigl( \abs{Q}^{-\alpha} \omega(Q)^{1/q} \sigma(Q)^{1/p'} \bigr)^{q \cdot (1+ \delta (1/q - 1/r))} \abs{Q}^{\alpha(-q + \delta q/r - \delta) + \alpha(q+\delta(1-q/r))} \\
			& \qquad \cdot \sigma(Q)^{-q/p'\cdot(1+\delta(1/q-1/r)) + q - \delta/p'(q/r-1)} \\
			& \le [\omega, \sigma]^{q (1+ \delta (1/q - 1/r))} \sigma(Q)^{q/p}.
    	\end{align*}
    	For the last estimate we need $\delta \bigl(\frac{1}{q} - \frac{1}{r} \bigr) \ge -1$. For $q > r$ this is equivalent to 
		\begin{equation*}
			\delta \le \frac{1}{\frac{1}{r} - \frac{1}{q}} = r \cdot \frac{1}{1-\frac{r}{q}}.
		\end{equation*}
		The second factor on the right is bigger than $1$ and therefore we find a suitable $\delta$ satisfying our new restriction together with~\eqref{eq:delta_assumptions}.
		Putting all together and using $\delta \bigl( 1- \frac{r}{q} \bigr) + r \bigl( 1+ \delta \bigl( \frac{1}{q} - \frac{1}{r} \bigr) \bigr) = r$ for the power of $[\omega, \sigma]$, we obtain by Lemma~\ref{lem:measure_estimate}
    	\begin{align*}
    		\MoveEqLeft \bignorm{\sum_{Q \in \mathcal{S}: Q \subset R} \abs{Q}^{-\alpha r} \sigma(Q)^r \mathds{1}_Q}_{L^{q/r}_{\omega}} \lesssim [\omega,\sigma]^r \biggl( \sum_{Q \in \mathcal{S}: Q \subset R} \sigma(Q)^{q/p} \biggr)^{r/q} \\
	& \leq [\omega,\sigma]^r \biggl( \sigma(R)^{q/p-1}\sum_{Q \in \mathcal{S}: Q \subset R} \sigma(Q) \biggr)^{r/q} \\
		 & \lesssim [\omega,\sigma]^r \biggl( \sigma(R)^{q/p-1}[\sigma]_{A_\infty}\sigma(R) \biggr)^{r/q} \leq [\omega,\sigma]^r[\sigma]_{A_\infty}^{r/q}\sigma(R)^{r/p}.
    	\end{align*}
    	 
	 	This finishes the estimate for $\mathcal{T}$ in the case $q > r$. We let $\mathcal{T}_r$ denote the value of $\mathcal{T}$ for a particular choice of $r$. For the case $q > p$ or $\alpha = 1$ now suppose that $q \le r$. We choose some $0 < s < q \le r$. We then have
		\begin{equation}
			\label{eq:reduction_testing_constant}
			\begin{split}
				\mathcal{T}_r^{1/r} & = \sup_{R \in \mathcal{S}} \sigma(R)^{-1/p} \biggnorm{\biggl( \sum_{Q \in \mathcal{S}: Q \subset R} \abs{Q}^{-\alpha r} \sigma(Q)^{r} \mathds{1}_Q \biggr)^{1/r}}_{L^q_{\omega}} \\
				& \le \sup_{R \in \mathcal{S}} \sigma(R)^{-1/p} \biggnorm{\biggl( \sum_{Q \in \mathcal{S}: Q \subset R} \abs{Q}^{-\alpha s} \sigma(Q)^{s} \mathds{1}_Q \biggr)^{1/s}}_{L^q_{\omega}} \\
				& \le \mathcal{T}_s^{1/s} \lesssim [\omega, \sigma] [\sigma]_{A_{\infty}}^{1/q}.
			\end{split}
		\end{equation}
		Taking both sides of the inequality to the power $r$ gives the desired estimate. Let us come to the very last case, namely $q \le r$, $p = q$ and $\alpha < 1$. Here we start with the special choice $r = q = p$. Then by~\eqref{eq:estimate_use_two_weight} for $\delta = r$
		\begin{align*}
			\MoveEqLeft \biggnorm{\sum_{Q \in \mathcal{S}: Q \subset R} \abs{Q}^{-\alpha r} \sigma(Q) \mathds{1}_Q}_{L^{q/r}_{\omega}} = \sum_{Q \in \mathcal{S}: Q \subset R} \abs{Q}^{-\alpha r} \sigma(Q)^{r} \omega(Q) \\
			& \le [\omega, \sigma]^r \sum_{Q \in \mathcal{S}: Q \subset R} \sigma(Q) \lesssim [\omega,\sigma]^r [\sigma]_{A_{\infty}} \sigma(R).
		\end{align*}
		Hence, $\mathcal{T}_q \lesssim [\omega, \sigma]^r [\sigma]_{A_{\infty}}$. Now, if $q < r$, then we choose $s = p = q$, and the same reasoning as in~\eqref{eq:reduction_testing_constant} gives
		\begin{equation*}
			\mathcal{T}_r^{1/r} \le \mathcal{T}_q^{1/q} \lesssim [\omega, \sigma]_{A_{pp}^{\alpha}} [\sigma]_{A_{\infty}}^{1/p}.
		\end{equation*}

		\textbf{Part II:} We now come to the estimate for the testing constant $\mathcal{T}^*$. Recall that here we are only interested in the case $p > r$. We write $s = (p/r)'$. We use Lemma~\ref{lem:norm_dyadic} again to rewrite the involved norm as
    	\begin{align*}
    		\MoveEqLeft \bignorm{\sum_{Q \in \mathcal{S}: Q \subset R} \abs{Q}^{-\alpha r} \sigma(Q)^{r-1} \omega(Q) \mathds{1}_Q}_{L^{(p/r)'}_{\sigma}} \\
    		& \simeq \biggl( \sum_{Q \in \mathcal{S}: Q \subset R} \sigma(Q) \abs{Q}^{-\alpha r} \sigma(Q)^{r-1} \omega(Q) \\
			& \qquad \cdot  \Big[\sigma(Q)^{-1} \sum_{Q' \subset Q} \abs{Q'}^{-\alpha r} \sigma(Q')^{r-1} \omega(Q') \sigma(Q') \Big]^{s - 1} \biggr)^{1/s} \\
    		& = \biggl( \sum_{Q \in \mathcal{S}: Q \subset R} \abs{Q}^{-\alpha r} \sigma(Q)^{r} \omega(Q) \Big[\sigma(Q)^{-1} \sum_{Q' \subset Q} \abs{Q'}^{-\alpha r} \sigma(Q')^{r} \omega(Q') \Big]^{s - 1} \biggr)^{1/s}.
    	\end{align*}
    	Observe that the inner sum $\sum_{Q' \subset Q} \abs{Q'}^{-\alpha r} \sigma(Q')^r \omega(Q')$ is exactly the same sum as in the first part of the proof. Hence, exactly the same considerations and estimates apply here. Again ignoring the additional constants appearing in the case $\alpha < 1$ and $p = q$, we get
    	\begin{align*}
    		\MoveEqLeft \bignorm{\sum_{Q \in \mathcal{S}: Q \subset R} \abs{Q}^{-\alpha r} \sigma(Q)^{r-1} \omega(Q) \mathds{1}_Q}_{L^{(p/r)'}_{\sigma}} \\
    		& \lesssim [\omega,\sigma]^{\delta/s'} \biggl( \sum_{Q \in \mathcal{S}: Q \subset R} \abs{Q}^{-\alpha r} \sigma(Q)^{r} \omega(Q) \\
			& \qquad \cdot \bigl(\sigma(Q)^{-1} \sigma(Q)^{r-\delta/p'} \abs{Q}^{\alpha (\delta-r)} \omega(Q)^{1-\delta/q} \bigr)^{s - 1} \biggr)^{1/s} \\
    		& = [\omega,\sigma]^{\delta r/p} \biggl( \sum_{Q \in \mathcal{S}: Q \subset R} \abs{Q}^{\alpha (-r + (\delta-r)(s-1))} \sigma(Q)^{r+(r-1-\delta/p')(s-1)} \\
			& \qquad \cdot \omega(Q)^{1+(1-\delta/q)(s-1)} \biggr)^{1/s} \\
			& = [\omega, \sigma]^{\delta r/p} \biggl( \sum_{Q \in \mathcal{S}: Q \subset R} \abs{Q}^{\alpha(-rs+\delta(s-1))} \sigma(Q)^{rs-(1+\delta/p')(s-1)} \\
			& \qquad \cdot \omega(Q)^{s-\delta/q(s-1)} \biggr)^{1/s}.
    	\end{align*}
    	In the following paragraph we will verify the following estimate step by step:
    	\begin{equation}
			\label{eq:ugly_powers}
			\begin{split}
    			\MoveEqLeft \abs{Q}^{\alpha(-rs+\delta(s-1))} \sigma(Q)^{rs-(1+\delta/p')(s-1)} \omega(Q)^{s-\delta/q(s-1)} \\
				& = \bigl( \abs{Q}^{-\alpha} \sigma(Q)^{1/p'} \omega(Q)^{1/q} \bigr)^{p' \cdot (rs-(1+\delta/p')(s-1))} \\
				& \qquad \cdot \abs{Q}^{\alpha(-rs+\delta(s-1)) + \alpha p' (rs-(1+\delta/p')(s-1))} \\
				& \qquad \cdot \omega(Q)^{s-\delta/q(s-1)-p'/q(rs-(1+\delta/p')(s-1))} \\
				& \lesssim [\omega, \sigma]^{p' (rs-(1+\delta/p')(s-1))} \omega(Q)^{s/q(q-r)},
			\end{split}
    	\end{equation}
		which holds provided we can find $\delta > 0$ for which the power of $[\omega, \sigma]$ is non-negative. For this consider the following:
		\begin{align*}
			\MoveEqLeft rs - \biggl( 1 + \frac{\delta}{p'} \biggr)(s-1) \ge 0 \Leftrightarrow r - \biggl( 1 + \frac{\delta}{p'} \biggr) \frac{r}{p} \ge 0 \\
			& \Leftrightarrow 1 + \frac{\delta}{p'} \le p \Leftrightarrow \delta \le p'(p-1) = p.
		\end{align*}
    	Note that we can find $\delta > 0$ that additionally satisfies the condition $\delta \le p$ because of the assumption $p > r$. Let us verify the identities for the powers of $\abs{Q}$ and $\omega(Q)$ used in~\eqref{eq:ugly_powers}. For the power of $\abs{Q}$ we have
    	\begin{align*}
    		\MoveEqLeft -rs + \delta(s-1) + p' \biggl( rs - \biggl( 1 + \frac{\delta}{p'} \biggr)(s-1) \biggr) = -rs + p'rs - p'(s-1) \\
			& = s(p'(r-1) - r) + p' = \frac{p(p(r-1) - r(p-1)) + p(p-r)}{(p-1)(p-r)} = 0,
    	\end{align*}
		whereas for the power of $\omega(Q)$ the following calculation shows the claimed identity
    	\begin{align*}
			\MoveEqLeft s - \frac{p'}{q} (rs - (s-1)) = s \biggl( 1 - \frac{p'}{q} \bigl( r - \bigl( 1 - \frac{1}{s} \bigr) \bigr) \biggr) = s \biggl( 1 - \frac{p'}{q} \bigl( r  - \frac{1}{s'} \bigr) \biggr) \\
			& =  s \biggl( 1 - \frac{rp'}{q} \bigl( 1 - \frac{1}{p} \bigr) \biggr) = s \biggr( 1 - \frac{r}{q} \biggr) = s \biggl( \frac{q-r}{q} \biggr).
    	\end{align*}    	
    	Hence, we have as desired
    	\begin{align*}
    		\MoveEqLeft \bignorm{\sum_{Q \in \mathcal{S}: Q \subset R} \abs{Q}^{-\alpha r} \sigma(Q)^{r-1} \omega(Q) \mathds{1}_Q}_{L^{(p/r)'}_{\sigma}} \\
			& \lesssim [\omega, \sigma]^{\delta r/p + p'/s\cdot (rs-(1+\delta/p')(s-1))} \biggl( \sum_{Q \in \mathcal{S}: Q \subset R} \omega(Q)^{s/(q/r)'} \biggr)^{1/s} \\
			& = [\omega, \sigma]^{r}  \biggl(\omega(R)^{s/(q/r)'-1} \sum_{Q \in \mathcal{S}: Q \subset R} \omega(Q) \biggr)^{1/s} \\
			& \lesssim [\omega, \sigma]^{r}  \biggl(\omega(R)^{s/(q/r)'-1} [\omega]_{A_\infty}\omega(R) \biggr)^{1/s}
    		 = [\omega, \sigma]^{r} [\omega]_{A_{\infty}}^{1/s} \omega(R)^{1/(q/r)'}.
    	\end{align*}
    	Let us again verify the power of $[\omega, \sigma]$ explicitly by a small calculation. We have
		\begin{align*}
			\MoveEqLeft \delta \frac{r}{p} + \frac{p'}{s} \biggl(rs- \biggr(1+ \frac{\delta}{p'} \biggr)(s-1) \biggr) = \delta \frac{r}{p} - \delta \biggl( 1 - \frac{1}{s} \biggr) + p' \biggl( r - \biggl(1 - \frac{1}{s} \biggr) \biggr) \\
			& = p' \biggl( r - \frac{r}{p} \biggr) = r. \qedhere
		\end{align*}
    \end{proof}

\section{The fractional square function}\label{sec:fractional_square_function}
	
	We now specialize our findings to the case of classical fractional square functions, i.e.\ $\alpha \in (0,1)$ and $r = 2$ for the condition $\alpha = \frac{1}{p'} + \frac{1}{q}$. In the one-weighted theory one here considers estimates for $L^p_{\omega^p} \to L^q_{\omega^q}$. For a weight $\omega\colon \IR^n \to \IR_{\ge 0}$, $p, q \in (1,\infty)$ and $\alpha \in (0,1]$ one is interested in sharp estimates in the one-weight characteristic
	\begin{equation*}
		[\omega]_{A_{pq}^{\alpha}} \coloneqq \sup_Q \biggl( \frac{1}{\abs{Q}} \int_Q \omega^q \biggr) \biggl( \frac{1}{\abs{Q}} \int_Q \omega^{-p'} \biggr)^{q/p'}.
	\end{equation*}
	Its relation to the two-weight characteristic is $[\omega^q, \omega^{-p'}]_{A_{pq}} = [\omega]_{A_{pq}}^{1/q}$. Hence, Theorem~\ref{thm:main} for $\sigma = \omega^{-1/(p-1)}$ gives the following mixed $A_{pq}-A_{\infty}$ estimate.
	
	\begin{corollary}
		Let $\alpha \in (0,1)$ and $1 < p \le q < \infty$ with $\frac{1}{q} + \frac{1}{p'} = \alpha$. Then
    	\begin{equation*}
      		\normalnorm{A^{2,\alpha}_{\mathcal{S}}}_{L^p_{\omega^p} \to L^q_{\omega^q}} \lesssim [\omega]^{\frac{1}{q}}_{A_{pq}} ( [\omega^{-p'}]_{A_\infty}^{\frac{1}{q}} + [\omega^q]_{A_\infty}^{(\frac{1}{2} - \frac{1}{p})_+}).
		\end{equation*}
	\end{corollary}
	
	One has $[\omega^q]_{A_{1+q/p'}} = [\omega]_{A_{pq}}$ and $[\omega^{-p'}]_{A_{1+p'/q}} = [\omega]_{A_{pq}}^{p'/q}$. In particular, $\omega \in A_{pq}$ implies the finiteness of the above $A_{\infty}$-characteristics. Using this relation to the $A_{pq}$-characteristic we a fortiori obtain the following pure $A_{pq}$-estimate.
	
	\begin{corollary}
		Let $\alpha \in (0,1)$ and $1 \le p \le q < \infty$ with $\frac{1}{q} + \frac{1}{p'} = \alpha$. Then
    	\begin{equation*}
      		\normalnorm{A^{2,\alpha}_{\mathcal{S}}}_{L^p_{\omega^p} \to L^q_{\omega^q}} \lesssim [\omega]_{A_{pq}}^{\frac{1}{q}} ( [\omega]_{A_{pq}}^{\frac{p'}{q^2}} + [\omega]_{A_{pq}}^{(\frac{1}{2} - \frac{1}{p})_+} ) \lesssim [\omega]_{A_{pq}}^{\max(\frac{p'}{q} \alpha, \alpha - \frac{1}{2})}. 
		\end{equation*}
	\end{corollary}
	
	This estimate is optimal in the following sense.
	
	\begin{proposition}
		Let $\alpha \in (0,1)$ and $1 < p \le q < \infty$ with $\frac{1}{q} + \frac{1}{p'} = \alpha$. If $\Phi\colon [1, \infty) \to \IR_{> 0}$ is a monotone function with
		\begin{equation*}
			\normalnorm{A^{2,\alpha}_{\mathcal{S}}}_{L^p_{\omega^p} \to L^q_{\omega^q}} \lesssim \Phi([\omega]_{A_{pq}})
		\end{equation*}
		for all $\omega \in A_{pq}$ and an implicit $\mathcal{S}$-dependent constant, then $\Phi(t) \gtrsim t^{\max(\frac{p'}{q} \alpha, \alpha - \frac{1}{2})}$.
		
	\end{proposition} 
	\begin{proof}
    	For $k \in \IN_0$ choose $I_k = [0,2^{-k}]$. Then the family $\mathcal{S} = (I_k)_{k \in \IN_0}$ is $\frac{1}{2}$-sparse. We consider its associated operator $A = A^{2,\alpha}_{\mathcal{S}}$. Following~\cite[Section~7]{LMPT10}, for $\epsilon \in (0,1)$ we let  $\omega_{\epsilon}(x) = \abs{x}^{(1-\epsilon)/p'}$ and $f(x) = \abs{x}^{\epsilon -1} \mathds{1}_{[0,1]}$. Then $[\omega_{\epsilon}]_{A_{pq}} \simeq \epsilon^{-q/p'}$ and $\norm{\omega_{\epsilon} f}_{L^p} \simeq \epsilon^{-1/p}$. Now, let $x \in [2^{-(k+1)}, 2^{-k}]$ for $k \in \IN_0$. Then
    	\begin{align*}
    		A f(x) \ge \abs{I_k}^{-\alpha} \int_{I_k} \abs{y}^{\epsilon - 1} \d y \mathds{1}_{I_k}(x)  = 2^{\alpha k} \int_{0}^{2^{-k}} \abs{y}^{\epsilon - 1} \d y  \simeq \abs{x}^{-\alpha} \cdot \epsilon^{-1} \abs{x}^{\epsilon}.
    	\end{align*}
    	Consequently,
    	\begin{align*}
    		\int_{\IR} Af^q \omega_{\epsilon}^q & \ge \sum_{k=0}^{\infty} \int_{2^{-(k+1)}}^{2^{-k}} Af^q \omega_{\epsilon}^q \ge \epsilon^{-q} \int_0^1 \abs{x}^{q(\epsilon-\alpha)} \abs{x}^{\frac{q}{p'}(1-\epsilon)} \d x \\
		& = \epsilon^{-q} \int_0^1 \abs{x}^{\frac{\epsilon}{p} - 1} \d x \simeq \epsilon^{-1-q}. 
    	\end{align*}
    	This shows that $\norm{Af}_{L^q_{\omega_{\epsilon}^q}} \gtrsim \epsilon^{-1-1/q}$. Hence, we get $\Phi(\epsilon^{-q/p'}) \gtrsim \epsilon^{-1-1/q + 1/p} = \epsilon^{-(1/p' + 1/q)} = \epsilon^{-\alpha}$. This finishes the first part of the estimate.
        	
	The second upper bound follows from a duality argument: If $A = {A}_{\mathcal{S}}^{2,\alpha}$ with $\mathcal{S}$ as above is reinterpreted as a vector-valued operator in a natural way, i.e., we have a bounded {\em linear} operator $A\colon L^p_{\omega^p}(\IR) \to L^q_{\omega^q}(\IR; \ell^2)$, then its adjoint with respect to the unweighted duality maps $L^{q'}_{\omega^{-q'}}(\IR; \ell^2)$ boundedly into $L^{p'}_{\omega^{-p'}}(\IR)$.
	Applying this adjoint to $(a_I\omega^q)_{I \in \mathcal{S}}$, for a sequence $(a_I)_{I \in \mathcal{S}}$ of measurable functions, one has the estimate
    	\begin{equation*}
    		\biggnorm{\sum_{I \in \mathcal{S}} \abs{I}^{-\alpha} \int_I a_I \omega^q \mathds{1}_I}_{L^{p'}_{\omega^{-p'}}} \lesssim \Phi([\omega]_{A_{pq}}) \biggnorm{\biggl( \sum_{I \in \mathcal{I}} \abs{a_I}^2 \biggr)^{1/2}}_{L^{q'}_{\omega^q}}.
    	\end{equation*}
    	Since the right hand side is independent of the sign of $a_I$, averaging and the Khintchine inequality give
    	\begin{equation}
			\label{eq:optimal_dual_estimate}
    		\biggnorm{\biggl( \sum_{I \in \mathcal{S}} \bigl( \abs{I}^{-\alpha} \int_{I} a_I \omega^q \bigr)^2 \mathds{1}_I \biggr)^{1/2}}_{L^{p'}_{\omega^{-p'}}} \lesssim \Phi([\omega]_{A_{pq}}) \biggnorm{\biggl( \sum_{I \in \mathcal{S}} \abs{a_I}^2 \biggr)^{1/2}}_{L^{q'}_{\omega^q}}.
    	\end{equation}
    	Now, for $\epsilon \in (0,1)$ choose $\omega_{\epsilon}(x) = \abs{x}^{(\epsilon - 1)/q}$. Then $[\omega]_{A_{pq}} = [\omega^q]_{1+q/p'} \simeq \epsilon^{-1}$. Further, imitating the argument in~\cite[Section~3]{LacScu12}, we choose $a_k(x) = \epsilon^{1/2} \abs{I_k}^{-\epsilon} \abs{x}^{\epsilon} \mathds{1}_{I_k}(x)$. With this, for $x \in (2^{-(l+1)}, 2^{-l}]$ and $l \in \IN_0$
    	\begin{align*}
    		\sum_{k=0}^{\infty} a_k^2(x) & = \epsilon \abs{x}^{2\epsilon} \sum_{k=0}^{\infty} \abs{I_k}^{-2\epsilon} \mathds{1}_{[0,2^{-k}]}(x) = \epsilon \abs{x}^{2\epsilon} \sum_{k=0}^l (2^{2 \epsilon})^{k} = \epsilon \abs{x}^{2\epsilon} \frac{2^{2(l+1)\epsilon} - 1}{2^{2\epsilon} - 1} \\
    		& \lesssim \abs{x}^{2\epsilon} 2^{2 l \epsilon} \lesssim 1.
    	\end{align*}
    	This directly gives for the right hand side of~\eqref{eq:optimal_dual_estimate}
    	\begin{equation*}
    		\biggnorm{\biggl( \sum_{k=0}^{\infty} \abs{a_k}^2 \biggr)^{1/2}}_{L^{q'}_{\omega^q}} \lesssim \biggl( \int_0^1 \abs{x}^{\epsilon - 1} \d x \biggr)^{1/q'} = \epsilon^{-1/q'}.
    	\end{equation*}
    	Let us now come to the left hand side of~\eqref{eq:optimal_dual_estimate}. First,
    	\begin{align*}
    		\abs{I_k}^{-\alpha} \int_{I_k} a_{k} \omega^q = \epsilon^{1/2} 2^{k(\alpha + \epsilon)} \int_0^{2^{-k}} \abs{x}^{2\epsilon - 1} = \frac{1}{2} \epsilon^{-1/2} 2^{k(\alpha - \epsilon)}.
    	\end{align*}
    	Consequently, for $x \in [2^{-(l+1)},2^{-l})$, $l \in \IN_0$ and $\epsilon \in (0, \alpha_0]$ with $\alpha_0 < \alpha$
    	\begin{align*}
    		\sum_{k=0}^{\infty} (\abs{I_k}^{-\alpha} \int_{I_k} a_{k} \omega^q)^2 \mathds{1}_{I_k}(x) = \frac{1}{4}\epsilon^{-1} \sum_{k=0}^l 2^{2k(\alpha - \epsilon)}
		\gtrsim
		 \epsilon^{-1} (2^{l})^{2(\alpha - \epsilon)} \simeq \epsilon^{-1} \abs{x}^{-2(\alpha - \epsilon)}.
    	\end{align*}
    	Thus,
    	\begin{align*}
    		\MoveEqLeft \biggnorm{\biggl( \sum_{I \in \mathcal{S}} \bigl( \abs{I}^{-\alpha} \int_{I} a_I \omega^q \bigr)^2 \mathds{1}_I \biggr)^{1/2}}_{L^{p'}_{\omega^{-p'}}} 
		\gtrsim
		 \epsilon^{-\frac{1}{2}} \biggl( \int_0^1 \abs{x}^{-p'(\alpha - \epsilon)} \abs{x}^{-(\epsilon - 1) \frac{p'}{q}} \d x \biggr)^{1/p'} \\
    		& = \epsilon^{-\frac{1}{2}} \left( 1 - p'(\alpha - \epsilon + \frac{1}{q} (\epsilon - 1)) \right)^{1/p'} = (\frac{p'}{q'})^{1/p'} \epsilon^{-\frac{1}{2} - \frac{1}{p'}}.
    	\end{align*}
    	Hence, one has for sufficiently small $\epsilon$ the estimate $\Phi(\epsilon^{-1}) \gtrsim \epsilon^{-\frac{1}{2} - \frac{1}{p'} + \frac{1}{q'}} = \epsilon^{\frac{1}{2} - \alpha}$, which gives the desired bound by the monotonicity of $\Phi$.
	\end{proof}


\begin{thebibliography}{CACDPO17}

\bibitem[BFP16]{BFP}
Fr\'ed\'eric Bernicot, Dorothee Frey, and Stefanie Petermichl.
\newblock Sharp weighted norm estimates beyond {C}alder\'on-{Z}ygmund theory.
\newblock {\em Anal. PDE}, 9(5):1079--1113, 2016.

\bibitem[CACDPO17]{CCDO}
Jos\'e~M. Conde-Alonso, Amalia Culiuc, Francesco Di~Plinio, and Yumeng Ou.
\newblock A sparse domination principle for rough singular integrals.
\newblock {\em Anal. PDE}, 10(5):1255--1284, 2017.

\bibitem[COV04]{CasOrtVer04}
Carme Cascante, Joaquin~M. Ortega, and Igor~E. Verbitsky.
\newblock Nonlinear potentials and two weight trace inequalities for general
  dyadic and radial kernels.
\newblock {\em Indiana Univ. Math. J.}, 53(3):845--882, 2004.

\bibitem[CU17]{Cruz17}
David Cruz-Uribe.
\newblock Two weight inequalities for fractional integral operators and
  commutators.
\newblock In {\em Advanced courses of mathematical analysis {VI}}, pages
  25--85. World Sci. Publ., Hackensack, NJ, 2017.

\bibitem[CUM13a]{CruMoe13a}
David Cruz-Uribe and Kabe Moen.
\newblock A fractional {M}uckenhoupt-{W}heeden theorem and its consequences.
\newblock {\em Integral Equations Operator Theory}, 76(3):421--446, 2013.

\bibitem[CUM13b]{CruMoe13}
David Cruz-Uribe and Kabe Moen.
\newblock One and two weight norm inequalities for {R}iesz potentials.
\newblock {\em Illinois J. Math.}, 57(1):295--323, 2013.

\bibitem[CUMP12]{CMP12}
David Cruz-Uribe, Jos\'e Mar\'\i~a Martell, and Carlos P\'erez.
\newblock Sharp weighted estimates for classical operators.
\newblock {\em Adv. Math.}, 229(1):408--441, 2012.

\bibitem[HL12]{HytLac12}
Tuomas~P. Hyt\"onen and Michael~T. Lacey.
\newblock The {$A_p$}-{$A_\infty$} inequality for general
  {C}alder\'on-{Z}ygmund operators.
\newblock {\em Indiana Univ. Math. J.}, 61(6):2041--2092, 2012.

\bibitem[HL15]{HytLi15}
Tuomas~P. {Hyt{\"o}nen} and Kangwei {Li}.
\newblock {Weak and strong $A_p$-$A_\infty$ estimates for square functions and
  related operators}.
\newblock {\em ArXiv e-prints}, September 2015.
\newblock Proc. Amer. Math. Soc., to appear.

\bibitem[HP13]{HytPer13}
Tuomas~P. Hyt\"onen and Carlos P\'erez.
\newblock Sharp weighted bounds involving {$A_\infty$}.
\newblock {\em Anal. PDE}, 6(4):777--818, 2013.

\bibitem[HRT17]{HRT17}
Tuomas~P. Hyt\"onen, Luz Roncal, and Olli Tapiola.
\newblock Quantitative weighted estimates for rough homogeneous singular
  integrals.
\newblock {\em Israel J. Math.}, 218(1):133--164, 2017.

\bibitem[Hyt12]{Hyt12b}
Tuomas~P. Hyt\"onen.
\newblock The sharp weighted bound for general {C}alder\'on-{Z}ygmund
  operators.
\newblock {\em Ann. of Math. (2)}, 175(3):1473--1506, 2012.

\bibitem[Hyt14]{Hyt14}
Tuomas~P. Hyt\"onen.
\newblock The {$A_2$} theorem: remarks and complements.
\newblock In {\em Harmonic analysis and partial differential equations}, volume
  612 of {\em Contemp. Math.}, pages 91--106. Amer. Math. Soc., Providence, RI,
  2014.

\bibitem[Lac14]{LacII}
Michael~T. Lacey.
\newblock Two-weight inequality for the {H}ilbert transform: a real variable
  characterization, {II}.
\newblock {\em Duke Math. J.}, 163(15):2821--2840, 2014.

\bibitem[Lac17]{Lac17}
Michael~T. Lacey.
\newblock An elementary proof of the {$A_2$} bound.
\newblock {\em Israel J. Math.}, 217(1):181--195, 2017.

\bibitem[Ler11]{Ler11}
Andrei~K. Lerner.
\newblock Sharp weighted norm inequalities for {L}ittlewood-{P}aley operators
  and singular integrals.
\newblock {\em Adv. Math.}, 226(5):3912--3926, 2011.

\bibitem[Ler13]{Ler13}
Andrei~K. Lerner.
\newblock On an estimate of {C}alder\'on-{Z}ygmund operators by dyadic positive
  operators.
\newblock {\em J. Anal. Math.}, 121:141--161, 2013.

\bibitem[Ler16]{Ler16}
Andrei~K. Lerner.
\newblock On pointwise estimates involving sparse operators.
\newblock {\em New York J. Math.}, 22:341--349, 2016.

\bibitem[LL16]{LacLi16}
Michael~T. Lacey and Kangwei Li.
\newblock On {$A_p$}--{$A_\infty$} type estimates for square functions.
\newblock {\em Math. Z.}, 284(3-4):1211--1222, 2016.

\bibitem[LMPT10]{LMPT10}
Michael~T. Lacey, Kabe Moen, Carlos P\'erez, and Rodolfo~H. Torres.
\newblock Sharp weighted bounds for fractional integral operators.
\newblock {\em J. Funct. Anal.}, 259(5):1073--1097, 2010.

\bibitem[LPRR17]{LPRR}
Kangwei {Li}, Carlos {P{\'e}rez}, Israel~P. {Rivera-R{\'{\i}}os}, and Luz
  {Roncal}.
\newblock {Weighted norm inequalities for rough singular integral operators}.
\newblock {\em ArXiv e-prints}, January 2017.

\bibitem[LS12]{LacScu12}
Michael~T. {Lacey} and James {Scurry}.
\newblock {Weighted Weak Type Estimates for Square Functions}.
\newblock {\em ArXiv e-prints}, November 2012.

\bibitem[LS15]{LS:entropy}
Michael~T. Lacey and Scott Spencer.
\newblock On entropy bumps for {C}alder\'on-{Z}ygmund operators.
\newblock {\em Concr. Oper.}, 2:47--52, 2015.

\bibitem[LSSUT14]{LSSU}
Michael~T. Lacey, Eric~T. Sawyer, Chun-Yen Shen, and Ignacio Uriarte-Tuero.
\newblock Two-weight inequality for the {H}ilbert transform: a real variable
  characterization, {I}.
\newblock {\em Duke Math. J.}, 163(15):2795--2820, 2014.

\bibitem[LSU09]{LacSawUri10}
Michael~T. {Lacey}, Eric~T. {Sawyer}, and Ignacio {Uriarte-Tuero}.
\newblock {Two Weight Inequalities for Discrete Positive Operators}.
\newblock {\em ArXiv e-prints}, November 2009.

\bibitem[Neu83]{Neug}
C.~J. Neugebauer.
\newblock Inserting {$A_{p}$}-weights.
\newblock {\em Proc. Amer. Math. Soc.}, 87(4):644--648, 1983.

\bibitem[P{\'e}r94a]{Per94b}
C.~P{\'e}rez.
\newblock Weighted norm inequalities for singular integral operators.
\newblock {\em J. London Math. Soc. (2)}, 49(2):296--308, 1994.

\bibitem[P{\'e}r94b]{Per94a}
Carlos P{\'e}rez.
\newblock Two weighted inequalities for potential and fractional type maximal
  operators.
\newblock {\em Indiana Univ. Math. J.}, 43(2):663--683, 1994.

\bibitem[RS16]{RS:entropy}
Robert Rahm and Scott Spencer.
\newblock Entropy bump conditions for fractional maximal and integral
  operators.
\newblock {\em Concr. Oper.}, 3:112--121, 2016.

\bibitem[Saw82]{Saw82}
Eric~T. Sawyer.
\newblock A characterization of a two-weight norm inequality for maximal
  operators.
\newblock {\em Studia Math.}, 75(1):1--11, 1982.

\bibitem[Saw88]{Saw88}
Eric~T. Sawyer.
\newblock A characterization of two weight norm inequalities for fractional and
  {P}oisson integrals.
\newblock {\em Trans. Amer. Math. Soc.}, 308(2):533--545, 1988.

\bibitem[TV16]{TV:entropy}
Sergei Treil and Alexander Volberg.
\newblock Entropy conditions in two weight inequalities for singular integral
  operators.
\newblock {\em Adv. Math.}, 301:499--548, 2016.

\bibitem[{Zor}16]{ZK}
P.~{Zorin-Kranich}.
\newblock {$A_p$-$A_\infty$ estimates for multilinear maximal and sparse
  operators}.
\newblock {\em ArXiv e-prints}, September 2016.

\end{thebibliography}

\end{document}